\documentclass[12pt]{amsart}
\usepackage{amssymb,latexsym,eucal}

\usepackage{hyperref}
\usepackage[letterpaper,centering,text={6.5in,9in}]{geometry}
\usepackage{graphicx}
\usepackage{breqn}
\usepackage{amsaddr}

\usepackage{latexsym}
\usepackage[utf8]{inputenc}
\usepackage{amsmath}
\usepackage{amsthm}
\usepackage{bigints}
\usepackage{amsfonts}
\usepackage{amssymb}
\usepackage{epsfig}
\usepackage{nicefrac}
\usepackage[all]{xy}
\usepackage{graphicx}
\usepackage{enumerate}
\usepackage{mathtools}    
\DeclarePairedDelimiterX\setc[2]{\{}{\}}{\,#1 \;\delimsize\vert\; #2\,}
\usepackage{float}
\usepackage{pgfplots}

\usepackage{caption,subcaption}
\usepackage{tikz}
\usetikzlibrary{
intersections, arrows.meta,
automata,er,calc,
backgrounds,
mindmap,folding,
patterns,
decorations.markings,
fit,decorations.pathmorphing,
shapes,matrix,
positioning,
shapes.geometric,
arrows,through
}

\newtheorem{theorem}{Theorem}[section]

\newtheorem{proposition}[theorem]{Proposition}
\newtheorem{lemma}[theorem]{Lemma}

\newtheorem{question}[theorem]{Question}
\newtheorem{conjecture}[theorem]{Conjecture}
\theoremstyle{definition}

\newtheorem{remark}[theorem]{Remark}

\newtheorem{problem}[theorem]{Problem}
\newtheorem{definition}[theorem]{Definition}

\newcommand{\eqdef}{\;{:=}\;}

\newcommand\Symp{\operatorname{Symp}}

\def\bigmid{\ \rule[-3.5mm]{0.1mm}{9mm}\ }

\newcommand{\II}{{\mathbb I}}
\newcommand{\CC}{{\mathbb C}}

\newcommand{\RR}{{\mathbb R}}

\newcommand{\NN}{{\mathbb N}}

\newcommand{\mw}{{{\mathrm {M}}}}

\newcommand{\lmw}{{\mathrm {M}}_{\mathrm{Sp}}}
\newcommand{\nlmw}{\mathrm{M}_{{\Symp}}}

\newcommand{\C}{{\mathrm{Conv}}}
\newcommand{\sv}{{\overline{{\mathrm{Vol}}}}}

\pgfplotsset{compat=1.18}

\begin{document}

\title[Symplectic variations of convex bodies and the mean width]{Symplectic variations of convex bodies and the mean width}
\author{Jonghyeon Ahn and Ely Kerman${}^*$}
\address{Department of Mathematics,
University of Illinois Urbana-Champaign\\Urbana, IL, 61801, USA.}
\email{ja34@illinois.edu, ekerman@illinois.edu $\text{($*$-corresponding author)}$}
\thanks{This research was supported by a grant from the Simons Foundation and funds from the Campus Research Board of the University of Illinois Urbana-Champaign.}

\date{\today}

\begin{abstract}
In this work, we study convex bodies in $\RR^{2n}$ with the property that their mean width cannot be infinitesimally decreased by  symplectomorphisms. 
The common theme of our results is that toric symmetry is a preferred feature of convex bodies with this property. 
\end{abstract}

\maketitle

\section{Introduction}   

Symplectomorphisms of $\RR^{2n}$ preserve the volume of subsets of $\RR^{2n}$. They do not preserve other classical measurements of subsets such as the surface area or the mean width. For any such measurement,  a natural problem then is to characterize those subsets of $\RR^{2n}$ which are in optimal symplectic position in the sense that the measurement of the subset is the smallest among the measurements of all its images under a natural class of  symplectomprphisms. In this work, we consider this problem for the mean width. Our primary motivation for this choice is the relationship between the mean width and the symplectic capacity of convex bodies that was established by Artstein-Avidan and Ostrover in \cite{ao}.

Recall that the mean width of a convex body $K$ in $\RR^{d}$ is defined as
\begin{align*}
    \mw (K) =\int_{S^{d-1}}(h_K(u)+h_K(-u))\, d \sigma,
\end{align*}
where
\begin{align*}
    h_K (u) =\sup_{k \in K} \, \langle k,u \rangle
\end{align*}
is the {support function} of $K$, and $\sigma$ is the rotationally invariant probability measure on the unit sphere  $S^{d-1} \subset \RR^d$. The quantity $h_K(u)+h_K(-u)$ being averaged is the caliper width of $K$ in the $u$ direction. This definition extends to any bounded subset $X \subset \RR^d$ and, denoting the convex hull of $X$ by $\C(X)$, one has $\mw(\C(X)) = \mw(X)$.

The general problem considered here is the following.
\begin{problem}\label{pr2}
Characterize the convex bodies $K \subset \RR^{2n}$ that satisfy
\begin{align}\label{p3}
  \inf_{\phi \in \Symp} \left\{\mw (\phi (K))\right\} =\mw (K),
\end{align}
where $\Symp$ is the group of symplectomorphisms of $\RR^{2n}$.
\end{problem}
We also consider the version of this problem involving only linear symplectic maps.
\begin{problem}\label{pr1}
Characterize  the convex bodies $K \subset \RR^{2n}$ that satisfy
\begin{align}\label{p1}
  \inf_{P \in \mathrm{Sp}
(2n)} \mw (PK)  = \mw (K),
\end{align}
where $\mathrm{Sp}(2n)$ is the group of symplectic matrices.
\end{problem}

An important first step here is to identify the convex bodies whose mean width cannot be decreased by (linear) symplectic maps near the identity. A subset of $\RR^{2n}$ is said to be $\mathit{toric}$ if it is invariant under the standard Hamiltonian action of the $n$-dimensional torus on $\RR^{2n} =\CC^n$ given by
$$(\theta_1, \dots, \theta_n) \cdot (z_1, \dots,z_n) = \left(e^{i\theta_1}z_1, \dots, e^{i \theta_n} z_n\right).$$
Our main result establishes that toric subsets are symplectic critical points of the mean width in the following sense.

\begin{theorem}\label{variation}
If $K$ is a toric convex domain, then for any smooth path $\phi^t$ in $\Symp$ with $\phi^0 = \mathrm{Id}$, $t=0$ is a critical point of the function
\begin{align*}
   t \mapsto \mw(\phi^t(K)).
\end{align*}
In fact, the derivative $\left.\frac{d}{dt}\right|_{t=0} M(\phi^t(K))$
 exists and is equal to zero.
\end{theorem}

For linear symplectic maps we prove that toric convex bodies are local minima for the mean width in the following sense.

\begin{theorem}\label{linear}
If $K$ is a toric convex domain, then the identity matrix  $\II \in \mathrm{Sp}(2n)$ is a local minimum of the function $\mathrm{Sp}(2n) \to \RR$ defined by 
\begin{align*}
 P \mapsto \mw(PK).
\end{align*}
Moreover, if $K$ is a toric convex body which is strictly convex and has a $C^2$-smooth boundary, then $\II \in \mathrm{Sp}(2n)$ is an isolated local minimum.
\end{theorem}

\begin{remark}\label{toric}
The mean width of a subset $X \subset \RR^d$ is unchanged by the action of the orthogonal group, $\mathrm{O}(d)$.  Hence, in both Theorem \ref{variation} and Theorem \ref{linear}, the hypothesis that $K$ is toric can be replaced by the more general hypothesis that $QK$ is toric for some $Q$ in $\mathrm{U}(n) = \mathrm{Sp}(2n) \cap \mathrm{O}(2n)$. 
\end{remark}

We conjecture that the assertion of Theorem \ref{linear} is true globally rather than just near the identity.

\begin{conjecture}
    If $K \subset \RR^{2n}$ is a toric convex domain, then \begin{align*}
  \inf_{P \in \mathrm{Sp}
(2n)} \mw (PK)  = \mw (K).
\end{align*}
\end{conjecture}

\begin{remark}An early version of this paper contained a mistaken proof of this assertion. The authors are grateful to the anonymous referee for pointing out the crucial error. This conjecture remains open even when $K$ is a symplectic ellipsoid or a symplectic polydisk in $\RR^4$. \end{remark}

\bigskip

One might hope that an analysis of the second variation might imply that toric domains are also local minima in the setting of  Theorem \ref{variation}. A first question in this direction is the following.

\begin{question}\label{second} Suppose that $K \subset \RR^{2n}$ is toric, strictly convex and has a smooth boundary. Let $\phi^t$ be a smooth path in $\Symp$ that passes through the identity map at $t=0$ and satisfies $ \left.\frac{d}{dt}\right|_{t=0} \phi^t(p) \notin T_p (\partial K)$ for some $p \in \partial K$. Is it true that 
  \begin{align}\label{ge}
    \left.\frac{d^2}{dt^2}\right|_{t=0} \mw(\phi^t(K)) > 0?
\end{align}  

\end{question}

Unfortunately, the simplest example reveals that the strict inequality \eqref{ge} does not hold in general. Consider the  unit disc $B^2 \subset \RR^2$ and the Hamiltonian flow, $\phi^t_H$, of a smooth function $H \colon \RR^2 \to \RR$. A straightforward computation, in polar coordinates $(\theta,r)$, yields
\begin{align}\label{hess}
    \left.\frac{d^2}{dt^2}\right|_{t=0} \mw(\phi_H^t(B^2)) = \frac{1}{\pi}\int_0^{2\pi} \left( \left(\frac{\partial^2H}{\partial \theta^2}(\theta,1)\right)^2-\left(\frac{\partial H}{\partial \theta}(\theta,1)\right)^2 \right)\, d \theta.
\end{align} 
Wirtinger's inequality, applied to $\frac{\partial H}{\partial \theta}(\theta,1)$, implies that the right hand side of \eqref{hess} is nonnegative and is equal to zero if and only if \begin{align*} \frac{\partial H}{\partial \theta}(\theta,1) =k_1 \sin (\theta -k_2)\end{align*} for constants $k_1$ and $k_2$.  Given this, one can easily find a function $G$, say, $G(\theta, r) =r^2 \cos \theta$, whose Hamiltonian vector field is not tangent to $\partial B^2$ and whose Hamiltonian flow $\phi^t_G$, satisfies
$$
 \left.\frac{d^2}{dt^2}\right|_{t=0} \mw(\phi_G^t(B^2)) = 0.
$$

\subsection{Context}
Here we discuss some of the motivations for considering Problem \ref{pr2} and  Problem \ref{pr1} and describe some implications of Theorem \ref{variation} and Theorem \ref{linear}.

\subsubsection{Motivation 1: Best positions in convex geometry} One motivation for studying Problem \ref{pr1} comes from convex geometry where it is common to consider the optimization of measurements over the space of positions of a convex body $K \subset \RR^{d}$, that is, the images of $K$ under matrices in $\mathrm{SL}(d)$. This idea plays a crucial role in fundamental results such as Ball's reverse isoperimetric inequality from \cite{ball}.

In \cite{gr}, Green gives a complete characterization of convex bodies $K \subset \RR^2$ that satisfy $$\inf_{T \in \mathrm{SL}(2)} \mw (TK)  = \mw (K).$$ This yields the following complete answer to Problem \ref{pr1} in dimension two.

\begin{theorem}[Green, \cite{gr}]\label{green}
A convex body $K \subset \RR^2$ satisfies $$\inf_{P \in \mathrm{\mathrm{Sp}(2)}} \mw (PK)  = \mw (K)$$
if and only if $$\int_0^{2\pi} h_{K}(\theta)\cos 2 \theta \, d \theta =0=\int_0^{2\pi} h_{K}(\theta)\sin  2 \theta \, d \theta.$$
\end{theorem}

In \cite{gm}, Giannopoulos and Milman give a complete characterization of convex bodies $K \subset \RR^d$ that satisfy $$\inf_{T \in \mathrm{SL}(d)} \mw (TK)  = \mw (K).$$ Their work implies the following partial answer to Problem \ref{pr1}.
\begin{theorem}[Giannopoulos and Milman, \cite{gm}]\label{gm}
A convex body $K \subset \RR^{2n}$ satisfies $$\inf_{P \in \mathrm{\mathrm{Sp}(2n)}} \mw (PK)  = \mw (K)$$
if $$\int_{S^{2n-1}} h_{K}(u) \langle u,v\rangle^2 \, d \sigma = \frac{\mw (K)}{4n}$$
for every $v \in S^{2n-1}.$
\end{theorem}

A complete answer to Problem \ref{pr1} seems to be far off. The special role to be played by conditions like toric symmetry is already evident when one restricts Problem \ref{pr1} to the set of ellipsoids in standard position. We define this set as 
$$\mathcal{E}_{\mathrm{stnd}}^{2n} =\left\{E(\mathbf{a},\mathbf{b}) \mid \mathbf{a}=(a_1, \dots, a_n),\, \mathbf{b}=(b_1, \dots, b_n) \in\RR^n_{>0} \right\}$$ where 
\begin{align*}
    E(\mathbf{a},\mathbf{b})= \left\{ (x_1, \dots, x_n, y_1, \dots, y_n) \in \mathbb{R}^{2n} \bigmid  \sum_{i=1}^n \left(\frac{x_i^2}{a^2_i} +\frac{y_i^2}{b^2_i}\right)\leq 1\right\}.
\end{align*}
Note that the ellipsoids with toric symmetry are precisely those of the form  $E(\mathbf{a}, \mathbf{a})$.

\begin{proposition}\label{ellipsoid}
The identity matrix  $\II \in \mathrm{Sp}(2n)$  is a local minimum of the function $\mathrm{Sp}(2n) \to \RR$ defined by 
\begin{align*}
 P \mapsto \mw(PE(\mathbf{a}, \mathbf{b} ))
\end{align*} 
if and only if $\mathbf{a}= \mathbf{b}$. Moreover, $\II$ is an isolated local minimum of the function $P \mapsto \mw(PE(\mathbf{a}, \mathbf{a} ))$ for any $\mathbf{a} \in \RR^n_{>0}.$
\end{proposition}

As described in Remark \ref{toric}, toric symmetry is not a necessary condition for a subset to be in optimal linear symplectic position with respect to the mean width, locally, because the mean width is preserved by the action of $\mathrm{U}(n) \subset \mathrm{Sp}(2n) $. The following result demonstrates that there is more to the fact that toric symmetry is not necessary in the lenear case. 

\begin{proposition}\label{lag}
For the Lagrangian bidisk
$$\mathbf{P}_L = \left\{ (x_1, x_{2}, y_1, y_{2}) \in \mathbb{R}^{4} \bigmid  x_1^2 +x_2^2 \leq 1,\, y_1^2 +y_2^2 \leq 1\right\},$$ the identity matrix is an isolated  local minimum of the function  $\mathrm{Sp}(4) \to \RR$ defined by
\begin{align*}
 P \mapsto \mw(P \mathbf{P}_L).
\end{align*} 
Moreover, $\mathbf{P}_L$ is not equal to $QK$ for any toric convex body $K$ and any matrix $Q \in \mathrm{U}(2)$.
\end{proposition}

The following question remains unresolved.
\begin{question}
    Does there exist a smooth path $\phi^t$ in $\Symp$ with $\phi^0 = \mathrm{Id}$ such that $$\left.\frac{d}{dt}\right|_{t=0} M(\phi^t(\mathbf{P}_L)) \neq 0?$$
\end{question}

However, for global variations the work of Ramos from \cite{vr} implies the following.

\begin{proposition}\label{vr}(Ramos, \cite{vr})
There exists a symplectomorphism $\Phi$ of $\RR^{2n}$ such that
\begin{align*}
  \mw (\Phi(\mathbf{P}_L))<\mw (\mathbf{P}_L).
  \end{align*}
\end{proposition}

\subsubsection{Motivation 2: Symplectic refinements of Urysohn's inequality}



Define the normalized volume of a subset $X$ of $\RR^{2n}$ to be
$$\sv(X) \eqdef \left(\frac{\mathrm{Vol}(X)}{\mathrm{Vol}(B^{2n})} \right)^{\frac{1}{n}}.$$

The following inequality, relating the volume to the mean width, is due to to Urysohn.
\begin{theorem}[Urysohn's inequality]\label{ury} For every convex body $K \subset \RR^{2n}$, one has
\begin{align}\label{ur}
    \sv(K) \leq \frac{(\mw(K))^2}{4}
\end{align}
with equality if and only if $K$ is a ball. 
\end{theorem}
The symplectic invariance of the volume can be used to refine \eqref{ur} to
\begin{align}\label{ref1}
    \sv(K) \leq \frac{1}{4}\left(\inf_{P \in \mathrm{Sp} 
(2n)} \mw (PK)\right)^2 
\end{align}
and then, further, to
\begin{align}\label{ref2}
 \sv(K) \leq \frac{1}{4}\left(\inf_{\phi \in \Symp} \mw (\phi (K)) \right)^2.
\end{align}
Problems and \ref{pr2} and \ref{pr1}  can be viewed as the first steps towards a deeper understanding of these refinements. In what follows we set 
$$\lmw(K) \eqdef \inf_{P \in \mathrm{Sp} 
(2n)} \mw (PK)$$
and 
$$\nlmw(K) \eqdef \inf_{\phi \in \Symp} \mw (\phi (K)).$$

In dimension two, refinements \ref{ref1} and \ref{ref2} are both understood and both substantial. The mean width of a convex body in $K$ in $\RR^{2}$ is equal to the length of its boundary divided by $\pi$, and Urysohn's inequality is equivalent to the classical isoperimetric inequality,
$$\mathrm{Area}(K) \leq \frac{(\mathrm{Length}(\partial K))^2}{4\pi}.$$
If one views the right-hand term as a means to approximate the left-hand term, it is easy to see, say for $K= [0,\epsilon] \times [0, \frac{1}{\epsilon}]$, that this approximation can be arbitrarily bad. Refinement \eqref{ref1} mitigates this problem. In this setting,  refinement \eqref{ref1} looks like
$$\mathrm{Area}(K) \leq \frac{1}{4\pi}\left(\inf_{P \in \mathrm{SL}(2)}\mathrm{Length}(\partial (PK))\right)^2$$  and the work of Gustin, from \cite{gu},  implies the following {\em reverse isoperimetric inequality},
\begin{align*}
 \frac{1}{4\pi}\left(\inf_{P \in \mathrm{SL}(2)}\mathrm{Length}(\partial (PK))\right)^2 \leq \frac{3 \sqrt{3}}{\pi}\mathrm{Area}(K),   
\end{align*} 
where equality holds if and only if $K$ is an equilateral triangle. 

The stronger refinement \eqref{ref2} completely resolves this issue in dimension two. In particular, we have the equality 
\begin{align*} 
    \mathrm{Area}(K) =\frac{1}{4\pi}\left(\inf_{\phi \in \Symp} \mathrm{Length}(\partial (\phi(K)))\right)^2
\end{align*}
 since one can easily find a smooth area preserving map such that the image of $K$ is arbitrarily close to a ball with respect to the Hausdorff metric. Unsurprisingly, the nature of refinements \eqref{ref1} and \eqref{ref2} in higher dimensions is less clear.

These considerations suggests two natural questions concerning the possible existence of {\em symplectic reverse Urysohn inequalities}. 

\begin{question}\label{l} For $n>1$, is there a constant $\gamma_n$ such that 
$$\lmw (K)  \leq \gamma_n \sv(K)^{\frac{1}{2}}$$ for all convex bodies $K \subset \RR^{2n}$?
\end{question}

\begin{question}\label{c}  For $n>1$, is there a constant $\Gamma_n$ such that 
$$ \nlmw(K) \leq \Gamma_n \sv(K)^{\frac{1}{2}}$$ for all convex bodies $K \subset \RR^{2n}$?
\end{question}

\begin{remark}If one replaces $\mathrm{Sp}(2n)$, in Question \ref{l}, by the larger group $\mathrm{SL}(2n)$, then such inequalities were shown to hold for symmetric convex bodies by Figiel and Tomczak-Jaegermann in \cite{ft}.\footnote{The difference between these two cases arises as an issue to be circumvented in \cite{ao-n} (see \S 4 therein).} \end{remark}

\subsubsection{Motivation 3: Inequalities between the mean width and symplectic capacities} Let $c_{EH}$ denote the first Ekeland and Hofer capacity from \cite{eh1}. Here, we normalize $c_{EH}$ so that the capacity of the closed unit ball, $B^{2n} \subset \RR^{2n}$, is equal to one. Much is known about the relation between the capacity of a convex set and its volume. A fundamental result in this direction is the following.

\begin{theorem}[Artstein-Avidan, Milman and Ostrover, \cite{amo}]
  There is a universal constant $A_0$, independent of $n$, such that for every convex body $K \subset \RR^{2n}$, one has
\begin{align}\label{amo}
    c_{EH}(K) \leq A_0 \sv(K).
\end{align}
\end{theorem}
This result was inspired by Viterbo's long-standing conjecture that inequality \eqref{amo} should hold for $A_0=1$. Recently, in \cite{hko}, Haim-Kislev and Ostrover disproved Viterbo's conjecture with the construction of a counterexample. In particular, their work implies that the best universal constant $A_0$, above, must be greater than one.

The following relationship between $c_{EH}$ and mean width was established in \cite{ao}.
\begin{theorem}[Artstein-Avidan and Ostrover, \cite{ao}]
    For every convex body $K \subset \RR^{2n}$, one has
\begin{align}\label{ao}
    c_{EH}(K) \leq \frac{(\mw(K))^2}{4}.
\end{align}
For symmetric $K$, equality holds if and only if $K$ is a ball.
\end{theorem}
Taken together, these results imply  that there exist convex bodies $K$ in $\RR^{2n}$, for $n>1$, such that 
\begin{equation}\label{weird}\sv(K) <  c_{EH}(K) < \frac{\left(\nlmw(K)\right)^2}{4}.\end{equation}
This leaves open the possibility that there may be convex bodies $K$ for which $\frac{1}{4}\left( \nlmw(K) \right)^2$ (or even $(\mw(K))^2/4$) is closer to $c_{EH}(K)$ than $\sv(K).$

\medskip

There are also interesting relationships between $\nlmw$ and the embedding capacity
$$c^B(X) =\inf\{ R^2 \mid \text{ there is a symplectic embedding } \phi \colon X \hookrightarrow RB^{2n} \}.$$
The simplest of these is the following.

\begin{lemma} For every bounded subset $X\subset \RR^{2n}$, one has
    \begin{align}\label{pinch}
\sv(X)
  \leq \frac{(\nlmw(X))^2}{4} \leq c^B(X).
\end{align}
\end{lemma}

\begin{proof}
The second inequality in \eqref{pinch} follows easily from the monotonicity of the mean width. The first inequality follows from Urysohn's inequality. To see this, recall  that 
$\mw (\C(X)) =\mw (X)$ and $\sv(\C(X)) \geq \sv(X)$, where $\C (X)$ is the convex hull of $X$. Urysohn's inequality implies that, for all $X\subset \RR^{2n}$ and every $\phi \in \Symp$, we have 
$$
\frac{(\mw(\phi(X)))^2}{4} = \frac{(\mw(\C(\phi(X))))^2}{4} \geq \sv(\C(\phi(X)))) \geq \sv(\phi(X))) =\sv(X).
$$
Taking the infimum over $\phi \in \Symp$ we get 
$$
\sv(X)
  \leq \frac{(\nlmw(X))^2}{4}.
$$
\end{proof}

A more subtle feature of this relationship is the following assertion which implies that $\nlmw$ detects symplectic information whenever it has room to.
\begin{proposition}\label{oneside}
 If $c^B(X)$ is strictly greater than $\sv(X)$, then \begin{align*}
  \frac{(\nlmw(X))^2}{4}>\sv(X). 
  \end{align*}
\end{proposition}

\begin{proof}
    It suffices to show that  if
    \begin{align}\label{left}
  \frac{1}{4}\left( \inf_{\phi \in \Symp}  \mw(\phi(X))\right)^2=\sv(X), 
  \end{align}
    then $c^B(X) = \sv(X).$ Equation \eqref{left}, together with \eqref{ur}, implies that there is a sequence of symplectomorphisms $\phi_i$ such that the nonnegative sequence \begin{align}\label{lim1}
        \frac{(\mw(\phi_i(X)))^2}{4} - \sv(X)
    \end{align}
converges monotonically to zero. From this, and \eqref{ur} again, it follows that 
    \begin{align}\label{lim2}
        \lim_{i \to \infty}\left(\frac{(\mw(\C(\phi_i(X))))^2}{4} -\sv(\C(\phi_i(X)))\right) = 0.
    \end{align} Composing the symplectomorphisms $\phi_i$ with suitable translations, if necessary, we may assume that the origin is the center of mass of each $\C(\phi_i(X))$. Since their mean widths are bounded, it follows that the sequence  $\C(\phi_i(X))$ is bounded. The same is then true of the sequence $\phi_i(X)$ and so we may pass to a subsequence, $\phi_{i_j}(X)$, that converges with respect to the Hausdorff metric, $d_H$. 
    Invoking the stability of the classical quermassintegral inequalities, as described in \cite{schn} (pages 421-423), it follows from \eqref{lim2} that 
    \begin{align}\label{lim3}
        \lim_{j \to \infty}d_H(\C(\phi_{i_j}(X)), B_{\mw}(\C(\phi_{i_j}(X)))) =0.
    \end{align}
    Here, $B_{\mw}(\C(\phi_{i_j}(X)))$ is the ball with same mean width as $\C(\phi_{i_j}(X)))$ and with center at the Steiner point of $\C(\phi_{i_j}(X)))$ (see \cite{schn}, page 50). Since the mean width is invariant under the convex hull operation, \eqref{lim3} is equivalent to  
\begin{align*}
        \lim_{j \to \infty}d_H(\C(\phi_{i_j}(X)), B_{\mw}(\phi_{i_j}(X))) =0.
    \end{align*}
Composing the symplectomorphisms $\phi_{i_j}$ with suitable translations, if necessary, we may assume the origin is the Steiner point of $\phi_{i_j}(X)$. With this, we have  
    \begin{align*}
        \lim_{j \to \infty}d_H\left(\C(\phi_{i_j}(X)), \frac{\mw(\phi_{i_j}(X))}{2}B^{2n}\right) =0,
    \end{align*}
By \eqref{lim1}, it follows that for all $\epsilon>0$ there is an integer $J$, such that for all $j>J$ we have   
    \begin{align*}
       \C(\phi_{i_j}(X))\subset \left((1+\epsilon)(\sv(X))^{\frac{1}{2}}\right)B^{2n}.
    \end{align*}
Since this works for any $\epsilon>0$. we have $c^B(X) \leq \sv(X)$ and the proof is complete.
\end{proof}

\begin{remark}
One might also ask if $\frac{(\nlmw(X))^2}{4}=c^B(X)$ implies that  $\sv(X)=c^B(X)$. 
\end{remark}
\medskip

 \begin{remark}The hypothesis of Proposition \ref{oneside} is that there is no full symplectic packing of a ball by $X$. The implication, that there is a $\Delta(X)>0$ such that \begin{align}\label{no}
  \frac{(\mw(\phi(X)))^2}{4} \geq \sv(X) +\Delta(X) \quad \text{for all $\phi \in \Symp$}
  \end{align}
is a manifestation of symplectic rigidity, expressed in terms of the mean width. Consider the case when $\sv(X)<c^B(X)$ and $X$ is diffeomorphic to $B^{2n}$, e.g., $X$ is an ellipsoid. By \cite{dm}, there is volume preserving diffeomorphism $\psi$ such that $\psi(X)$ is a ball, and so
\begin{align*}
  \frac{(\mw(\psi(X)))^2}{4}=\sv(X). 
\end{align*} 
As described by Viterbo in \cite{vit-lp}, one can also construct a sequence of symplectomorphisms. $\phi_{j}$. that converges to $\psi$ in the $L^p$-topology. 
However, by \eqref{no} we have
\begin{align*}
  \frac{(\mw(\phi_{j}(X)))^2}{4}\geq \sv(X) + \Delta(X), \quad \text{for all $j$. }
\end{align*} 
The fact that the mean width sees the {\em symplectic gap} between the $\phi_j$ and $\psi$ is not obvious from the construction in \cite{vit-lp}.
\end{remark}

\noindent{\em Observaton: A mean width staircase.} For $a \geq 1$, consider the family of symplectic ellipsoids $$\mathbf{E}(a)
= \left\{ (x_1,x_2,y_1,y_2) \in \RR^{4} \bigmid  x_1^2 + y_1^2 + \frac{x_2^2 + y_2^2}{a}  \leq 1\right\}.$$ 
The inequalities of \eqref{pinch} imply that the graph of 
\begin{align}\label{graph}
 a \mapsto \frac{(\nlmw(\mathbf{E}(a))^2}{4}   
\end{align} lies between that of $\sv(\mathbf{E}(a)) = \sqrt{a}$ and the intricate Fibonnaci staircase computed by McDuff and Schlenk, in \cite{mcs}, that defines the graph of $c^B(\mathbf{E}(a))$. By Proposition \ref{oneside}, the graph of \eqref{graph} lies strictly above the graph of $\sqrt{a}$ whenever the Fibonnaci staircase does, see Figure \ref{fig:mws}. Hence it forms its own version of a staircase. This is distinct from the one from \cite{mcs}. In particular, for $a \in (1,2)$ the strict inequalities 
$$
\sv(\mathbf{E}(a) < \frac{(\nlmw(\mathbf{E}(a))^2}{4} < c^B(\mathbf{E}(a))
$$
hold. The first inequality is implied by Proposition \ref{oneside}, since $c^B(\mathbf{E}(a)) = a > \sqrt{a} = \sv(\mathbf{E}(a))$ for $a \in (1,2)$. The second strict inequality follows from a direct computation of $\mw(\mathbf{E}(a))$ since  $$\frac{(\nlmw(\mathbf{E}(a))^2}{4}\leq \frac{(\mw(\mathbf{E}(a))^2}{4}=\frac{4}{9} \left( \frac{1 +\sqrt{a}+a}{1+\sqrt{a}}\right)^2 <a.$$

\begin{figure}[H]
    \centering
    \caption{$\nlmw(\mathbf{E}(a))^2/4$ lies below $c^B(\mathbf{E}(a))$ and strictly above $\sqrt{a}$, when there is room. For $1<a<2$, it also lies below $\mw(\mathbf{E}(a))^2/4$.}
    \label{fig:mws}
    \vspace{.5cm}
    \resizebox{0.83\textwidth}{!}{
    \begin{tikzpicture}
    \tikzstyle{every node}=[font=\tiny]
    \draw[<->](0,4) -- (0,0) -- (8,0);
    \draw[teal, domain=1:7] plot (\x, {sqrt(\x)});
    \draw[teal, dotted, domain=0:1] plot (\x, {sqrt(\x)});
    \draw[blue, domain=1:2] plot(\x, {\x});
    \draw[blue, domain=2:4] plot (\x, {2});
    \draw[blue, domain=4:5] plot (\x, {0.5*\x});
    \draw[blue, domain=5:25/4]plot(\x, {5/2});
    \draw[blue, domain=25/4:13/2]plot(\x, {(2/5)*\x});
    \draw[dotted, domain=0:1]plot(\x,{1});
    \draw[dotted, domain=0:2]plot(\x,{2});
    \draw[dotted, domain=0:5]plot(\x,{5/2});
    \draw[dotted] (1,0) -- (1,1);
    \draw[dotted] (4,0) -- (4,2);
    \draw[dotted] (5,0) -- (5,5/2);
    \draw[dotted] (25/4,0) -- (25/4,5/2);
    \draw[red, domain=1:2] plot (\x, {0.7*\x+0.3});
    \draw[red, dotted, domain=2:2.3] plot (\x, {2*sqrt(\x)-1.13});
    \node at (1,-0.3) {1};
    \node at (-0.3, 1) {1};
    \node at (-0.3, 2) {2};
    \node at (-0.3, 5/2) {$\frac{5}{2}$};
    \node at (4, -0.3) {4};
    \node at (5, -0.3) {5};
    \node at (25/4, -0.3) {$\frac{25}{4}$};
    \node at (8, -0.3) {$a$};
    \node[blue] at (5, 3) {$c^B(\mathbf{E}(a))$};
    \node[teal] at (7 ,2) {$\overline{\mathrm{Vol}}(\mathbf{E}(a))=\sqrt{a}$};
    \node[red] at (3.5, 1) {$\displaystyle\frac{\mathrm{M}(\mathbf{E}(a))^2}{4}$};
\end{tikzpicture}
}
\end{figure}
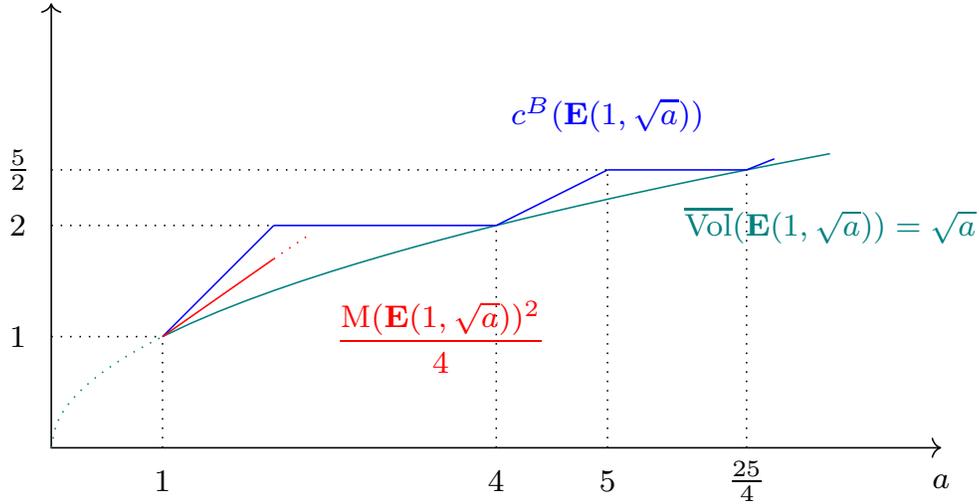

\subsection{Organization}
The proof of Theorem \ref{variation} is contained in Section \ref{sec:variation} and the proof of Theorem \ref{linear} is contained in Section \ref{sec:linear}. In both cases, it sufficed to consider toric convex domains whose boundary is smooth and strictly convex.  Section \ref{sec:ellipsoid} contains the proof of Proposition \ref{ellipsoid} concerning ellipsoids in standard position. Propositions \ref{lag} and \ref{vr} concerning the Lagrangian bidisk are proved in Section \ref{lag}. 

\section{Proof of Theorem \ref{variation}}\label{sec:variation}

\subsection{Simplifications}

We will first prove Theorem \ref{variation} for strictly convex bodies whose boundaries are smooth. Let $K$ be a convex body with these properties. The support function of $K$ can then be given by an explicit, and very useful, formula. In particular, the Gauss map of $K$,
$$\nu_K \colon \partial K \to S^{d-1},$$
which takes $u \in \partial K$ to its outward unit normal vector, is a diffeomorphism and we have
\begin{align*}
    h_K(u)=\langle u, \nu_K^{-1}(u)\rangle.
\end{align*}

We can also restrict our attention to paths of symplectomorphisms generated by autonomous Hamiltonian vector fields. In particular, it suffices to show that for any smooth and  compactly supported function $H\colon \RR^{2n} \to \RR$, we have
\begin{align*}
   \left.\frac{d}{dt}\right|_{t=0} M(\phi^t_H(K)) = 0
\end{align*}
where $\phi^t_H$ is the flow of the Hamiltonian vector field $X_H$ defined by $$\omega_{2n}(X_H(z), \cdot) = dH(z)[\cdot],$$ and $\omega_{2n}$ is the standard symplectic structure on $\RR^{2n}$.

Fixing  a function $H$ as above, there is an $\epsilon>0$ such that the images $K_t=\phi^t_H(K)$ are strictly convex for all $|t|\leq \epsilon$. Restricting to $t \in [-\epsilon, \epsilon]$, we then have 
$$M(K_t) = 2 \int_{S^{2n-1}} \langle u,\nu^{-1}_{K_t}(u)\rangle \, d \sigma.$$ To prove Theorem \ref{variation}, it now suffices to show that 
\begin{equation}\label{critneed}
\int_{S^{2n-1}} \left\langle u, \left.\frac{d}{dt}\right|_{t=0} \nu^{-1}_{K_t}(u)\right\rangle \, d \sigma =0.
\end{equation}

\subsection{Coordinate expressions} To prove \eqref{critneed}, we must take advantage of the assumption that $K$ is toric. Define the Hopf coordinates $(\theta_1, \cdots, \theta_n, r_1, \cdots, r_n)$ on $\RR^{2n}$ by
\begin{align*}
    x_i = r_i \cos{\theta_i}\,\,\,\text{and}\,\,\, y_i = r_i \sin{\theta_i} \,\,\,\text{for}\,\,\, i = 1, \cdots, n.
\end{align*}
Consider also the coordinate functions $(\theta,r,\rho)=(\theta_1, \cdots, \theta_n, r_1, \cdots, r_{n-1}, \rho)$ where
\begin{align*}
   \rho = \sqrt{r_1^2 + \cdots + r_n^2}.
\end{align*}
The unit sphere $S^{2n-1}\subset \RR^{2n}$ corresponds to $\rho=1$ and $$(\theta, r)=(\theta_1, \cdots, \theta_n, r_1, \cdots, r_{n-1})$$ are coordinate functions on $S^{2n-1}$ where the points $r =(r_1, \dots, r_{n-1})$ belong to the 
region $$\displaystyle{R=\left\{ \sum_{j=1}^{n-1} r_j^2 \leq 1\right\} \subset \RR^{n-1}_{\geq0}}.$$
Integration of a function $F$ over $S^{2n-1}$, in these coordinates, then becomes
\begin{align*}
    \int_{S^{2n-1}} F(u) \, d \sigma = \frac{(n-1)!}{2 \pi^n}\int_{R \times T^n} F(u(\theta,r))r_1 \cdots r_{n-1}d r_1 \cdots d r_{n-1} d \theta_1 \cdots d \theta_n
\end{align*}
where $T^n$ is the standard $n$-torus, $(\RR/2\pi\RR)^n$.  

Each convex domain $K \subset \RR^{2n}$, with the origin in its interior, is defined by an equation of the form $$\{\rho = f(\theta, r) \mid (r, \theta) \in R \times T^n\}.$$ The convex body $K$ is toric if and only if $f$ does not depend on $\theta$.

\begin{definition}
A smooth function $J \colon S^{2n-1} \to \RR$ is said to be $\theta$-{\em simple} if 
\begin{align*}
J(\theta,r) =& \sum_{i=1}^\ell A_i(r)\frac{\partial J_i}{\partial \theta_{\zeta(i)}}(\theta,r).
\end{align*}
where the functions $A_i(r)$ and $J_i(\theta,r)$ are smooth and $\zeta$ is a map from $\{1, \dots, \ell\}$ to $\{1, \dots, n\}$.
\end{definition}

\begin{remark}\label{simple} Note that $J$ and $J'$ are $\theta$-simple, then so is $jJ+J'$ for any smooth function $j\colon R \to \RR.$ \end{remark}

The point of this definition is the following simple observation.

\begin{lemma}\label{simple0}
If the smooth function  $J \colon S^{2n-1} \to \RR$ is $\theta$-simple, then 
$$\int_{S^{2n-1}} J \, d\sigma =0.$$
\end{lemma}

\begin{proof}
    Integration by parts yields

\begin{align*}
   & \int_{S^{2n-1}} J \, d\sigma\\=&\sum_{i=1}^{\ell}\int_{R \times T^n} A_i(r)\frac{\partial J_i}{\partial \theta_{\zeta(i)}}(\theta,r)\,r_1 \cdots r_{n-1}d r_1 \cdots d r_{n-1} d \theta_1 \cdots d \theta_n \\ =& \sum_{i=1}^{\ell}\int_{R \times T^n} \frac{\partial}{\partial \theta_{\zeta(i)}}\left(A_i(r)r_1 \cdots r_{n-1}\right) J_i(\theta,r)\,d r_1 \cdots d r_{n-1} d \theta_1 \cdots d \theta_n\\ =&0.
\end{align*}
\end{proof}
Given this, in order to prove \eqref{critneed}, it suffices to prove the following.

\begin{proposition}\label{form}
Let $K =\{\rho=f(r)\}$ be a strictly convex toric domain in $\RR^{2n}$ with smooth boundary. For any smooth function $H \colon \RR^{2n} \to \RR$ the function
\begin{align*}
\left\langle u, \left.\frac{d}{dt}\right|_{t=0} \nu^{-1}_{K_t}(u)\right\rangle
\end{align*}
is $\theta$-simple where $K_t =\phi^t_H(K)$.
\end{proposition}

\subsection{Proof of Proposition \ref{form}}
For all sufficiently small $|t|>0$, the boundaries of the domains $K_t =\phi^t_H(K)$ are strictly convex and their (smooth) boundaries are given by \begin{align*}
\partial K_t = \{\rho = f_t(\theta,r)\} = \{f_t(u)u \in \RR^{2n} \mid u \in S^{2n-1}\} 
\end{align*} for the smooth family of functions $f_t \colon S^{2n-1} \to \RR$ defined by $\rho(t) =f_t (\theta(t), r(t))$ where $(\theta(t),r(t),\rho(t)) =\phi^t_H(\theta,r,\rho)$.

The Gauss map of $K_t$ is the diffeomorphism from $\partial K_t$ to $S^{2n-1}$  given by
\begin{align*}
    \nu_{K_t}(f_t(u)u) = \left(\frac{\nabla(\rho - f_t)}{\|\nabla(\rho - f_t))\|}\right) (f_t(u)u).
\end{align*}
Here, $\nabla f_t$ is viewed as the gradient vector field, along $\partial K_t$, of the $\rho$ -independent function $f_t= f_t(\theta,r)$ defined near $\partial K_t$. It is also implicitly understood that the resulting unit vector, which in the formula formally lies in the tangent space $T_{f_t(u)u}\RR^{2n}$, has been identified in the obvious way with a point in $S^{2n-1}.$  

It will also be useful to consider the related diffeomorphisms of $S^{2n-1}$, 
\begin{equation}\label{w}
    w_{K_t}(u) = \left(\frac{\nabla(\rho - f_t)}{\|\nabla(\rho - f_t))\|}\right) (f_t(u)u).
\end{equation}
In particular, we have
\begin{align}\label{split}
 \nu_{K_t}^{-1}(u) = f_t(w_{K_t}^{-1}(u))w_{K_t}^{-1}(u). 
\end{align}
In terms of the coordinates $(\theta,r)$ on $S^{2n-1}$, the diffeomorphism $w_K$ has the form $w_K(\theta,r)= (\theta, G(r))$ where $G$ is a bijection of $R$. Hence, $w^{-1}_K(\theta,r)= (\theta, F(r))$ for $F=G^{-1}$.
It follows from Equation \eqref{split} that 
\begin{align}\label{master}
  \left\langle u, \left.\frac{d}{dt}\right|_{t-0} \nu^{-1}_{K_t}(u)\right\rangle =& \,f_0(w_K^{-1}(u))\left\langle u, \left.\frac{d}{dt}\right|_{t=0} w^{-1}_{K_t}(u)\right\rangle  +\dot{f}_0(w_{K}^{-1}(u))\langle u, w_K^{-1}(u)\rangle\\ &\nonumber + df_0 (w_{K}^{-1}(u)) \left[  \left.\frac{d}{dt}\right|_{t=0} w^{-1}_{K_t}(u)\right] \langle u, w_K^{-1}(u)\rangle, 
\end{align}
where $f_0=f$ and $\dot{f}_0(u) = \left.\frac{d}{dt}\right|_{t=0} f_t(u)$. It suffices to show that each of the three summands on the right side of \eqref{master} is $\theta$-simple. With this in mind, we now analyze their constituents starting with the simple observation.
\begin{lemma}\label{f0}
For all $u =(\theta,r) \in S^{2n-1}$, we have 
\begin{align*}
 f_0(w_K^{-1}(u)) = f_0(F(r)).   
\end{align*}
\end{lemma}

Writing $u \in S^{2n-1} \subset R^{2n}$ in Cartesian coordinates as $$u=(x_1, \dots, x_n,y_1, \dots, y_n)^T$$ we also have 
$$w^{-1}_K(u) = \mathrm{diag}(a_1(r),\dots,a_n(r),a_1(r),\dots,a_n(r)) u$$ for nonnegative functions $a_j$ on $S^{2n-1}$ which do not depend on $\theta$ and which satisfy \begin{align}\label{aj}\sum_{j=1}^{n-1} r_j^2 a_j^2(r)  + \left(1 -\sum_{j=1}^{n-1} r_j^2\right)a_n^2(r)=1.\end{align} This implies 
\begin{lemma}\label{2b}
For all $u =(\theta,r) \in S^{2n-1}$, we have 
\begin{align*}
    \left\langle u,w_K^{-1}(u) \right\rangle = \sum_{i=1}^{n-1} r_i^2 a_i(r) + \left(1-\sum_{i=1}^{n-1}r_i^2\right)a_n(r)
\end{align*}
for smooth functions $a_i \colon R \to \RR_{\geq 0}$ that satisfy \eqref{aj}.
\end{lemma}

We also have 
\begin{lemma}\label{D}
\begin{align*}
D w_K^{-1}(u)\left[\frac{\partial}{\partial x_i}\right] = \begin{cases}
    \displaystyle\cos{\theta_i} \sum_{j=1}^n r_j \cos{\theta_j}\frac{\partial a_j}{\partial r_i}\frac{\partial }{\partial x_j} + \cos{\theta_i} \sum_{j=1}^n r_j \sin{\theta_j}\frac{\partial a_j}{\partial r_i}\frac{\partial }{\partial y_j} +a_i(r) \frac{\partial}{\partial x_i}\,\,\text{for}\,\,i \neq n\\
    \displaystyle a_n(r)\frac{\partial}{\partial x_n} \,\,\text{for}\,\,i=n
\end{cases}
\end{align*}
and
\begin{align*}
D w_K^{-1}(u)\left[\frac{\partial}{\partial y_i}\right] =
    \begin{cases}
    \displaystyle\sin{\theta_i} \sum_{j=1}^n r_j \cos{\theta_j}\frac{\partial a_j}{\partial r_i}\frac{\partial }{\partial x_j} + \sin{\theta_i} \sum_{j=1}^n r_j \sin{\theta_j}\frac{\partial a_j}{\partial r_i}\frac{\partial }{\partial y_j}+ a_i(r) \frac{\partial}{\partial y_i}\,\,\text{for}\,\,i \neq n\\
        \displaystyle a_n(r)\frac{\partial}{\partial y_n} \,\,\text{for}\,\,i=n.
    \end{cases}
\end{align*}
\end{lemma}

Since the functions $f_t$ define the boundaries $\partial K_t = \phi^t_H(\partial K)$ by  $$\{\rho(t) = f_t(\theta(t), r(t))\}$$ where $(\theta(t), r(t), \rho(t))=\phi_H^t(\theta, r, \rho),$
 we have
\begin{align}\label{f.0}
    \dot{\rho}(0) = \dot{f_0}(\theta(0), r(0)) + (df_0)_{(\theta(0), (0))} [\dot\theta(0)] +(df_0)_{(\theta(0), r(0))} [\dot{r}(0)].
\end{align}
Using the formula for the Hamiltonian vector field $X_H$ of $H$,
\begin{align*}
    X_H =\, \sum_{i=1}^{n-1} \left( - \frac{1}{r_i} \frac{\partial H}{\partial r_i} - \frac{1}{\rho} \frac{\partial H}{\partial \rho}\right) \frac{\partial}{\partial \theta_i}-\frac{1}{\rho}\frac{\partial H}{\partial \rho} \frac{\partial}{\partial \theta_n}  + \sum_{i=1}^{n-1}\frac{1}{r_i}\frac{\partial H}{\partial \theta_j} \frac{\partial}{\partial r_i} +\frac{1}{\rho}\left( \sum_{j=1}^n \frac{\partial H}{\partial \theta_j}\right) \frac{\partial}{\partial \rho},
\end{align*}
equation \eqref{f.0} yields 
\begin{align*}
    \dot{f_0} =\frac{1}{f_0}\left( \sum_{j=1}^n \frac{\partial H}{\partial \theta_j}\right) + 
    \sum_{i=1}^{n-1} \left(\frac{1}{r_i} \frac{\partial H}{\partial r_i} + \frac{1}{f_0} \frac{\partial H}{\partial \rho} \right)\frac{\partial f_0}{\partial \theta_i} +  \frac{1}{f_0} \frac{\partial H}{\partial \rho}\frac{\partial f_0}{\partial \theta_n}-\sum_{i=1}^{n-1}\frac{1}{r_i}\frac{\partial H}{\partial \theta_j} \frac{\partial f_0}{\partial r_i}.
\end{align*}
where $f_0$ and its derivatives are evaluated at $(\theta,r)$ and  $H$ and its partial derivatives are evaluated at $(\theta,r, f_0(r))$.
Since $f_0(\theta,r) =f_0(r)$, this simplifies to 
\begin{align*}
    \dot{f_0} = \frac{1}{f_0}\left( \sum_{j=1}^n \frac{\partial H}{\partial \theta_j}\right)-\sum_{i=1}^{n-1}\frac{1}{r_i}\frac{\partial f_0}{\partial r_i}\frac{\partial H}{\partial \theta_j},
\end{align*}
 and evaluating this at $w_K^{-1}(u)$, we get

\begin{lemma}\label{2a} For all $u =(\theta,r) \in S^{2n-1}$, we have 
    \begin{align*}
    \dot{f_0}(w_K^{-1}(u)) = \frac{1}{f_0(F(r))} \sum_{j=1}^n \frac{\partial H}{\partial \theta_j}(\theta,F(r), f_0(F(r)))-\sum_{i=1}^{n-1}\frac{1}{F_i(r)}\frac{\partial f_0}{\partial r_i}(F(r))\frac{\partial H}{\partial \theta_j}(\theta,F(r), f_0(F(r)))
    \end{align*}
    where $F(r)=(F_1(r), F_2(r), \dots, F_{n-1}(r))$.
\end{lemma}

It remains for us to analyze the coordinate expression of $$\left.\frac{d}{dt}\right|_{t=0} w^{-1}_{K_t}(u).$$
From the identity $w_{K_t}(w_{K_t}^{-1}(u))=u$, we get 
\begin{align*}
     D w_{K}(w_{K}^{-1}(u)) \left[\left.\frac{d}{dt}\right|_{t=0} w_{K_t}^{-1} (u) \right] + \dot{w}_{K} (w_{K}^{-1}(u)) =0
\end{align*}
where $\dot{w}_{K} (v) =\left.\frac{d}{dt}\right|_{t=0} w_{K_t} (v)$.
Hence, we have  
\begin{align}\label{dinverse}
 \left.\frac{d}{dt}\right|_{t=0} w_{K_t}^{-1} (u) = -D w_{K}^{-1} (u) \left[ \dot{w}_{K} (w_{K}^{-1}(u)) \right].
\end{align}

First, we consider $\dot{w}_{K}.$ 
Setting  $$v_t(u) = \nabla(\rho-f_t)(f_t(u)u)$$ and viewing this as a point in $\RR^{2n}$, it follows from equation \eqref{w} that
\begin{align}\label{w2v}
    \dot{w}_{K}(u) = \frac{\|v_0(u)\|^{2} \dot{v}_0(u) - \langle \dot{v}_0(u),v_0(u) \rangle v_0(u)}{\|v_0(u)\|^{3}}.
\end{align}
The gradient vector of a smooth function $F=F(\theta,r,\rho)$, written in terms of Cartesian tangent vectors, is given by 
\begin{align*}
    \nabla F(\theta, r, \rho) =& \sum_{i=1}^{n-1}\left(-\frac{1}{r_i}\frac{\partial F}{\partial \theta_i}\sin \theta_i + \left(\frac{\partial F}{\partial r_i} + \frac{r_i}{\rho}
    \frac{\partial F}{\partial \rho}\right)\cos \theta_i \right)\frac{\partial}{\partial x_i} \\ \nonumber&+ \left(-\frac{1}{r_n(r,\rho)}\frac{\partial F}{\partial \theta_n}\sin \theta_n  +  \frac{r_n(r,\rho)}{\rho} \frac{\partial F}{\partial \rho}\cos \theta_n \right)\frac{\partial}{\partial x_n}  \\ \nonumber&+ \sum_{i=1}^{n-1}\left(\frac{1}{r_i}\frac{\partial F}{\partial \theta_i}\cos \theta_i + \left(\frac{\partial F}{\partial r_i} + \frac{r_i}{\rho}
    \frac{\partial F}{\partial \rho}\right)\sin \theta_i \right)\frac{\partial}{\partial y_i} \\ \nonumber&+ \left(\frac{1}{r_n(r,\rho)}\frac{\partial F}{\partial \theta_n}\cos \theta_n  + \frac{r_n(r,\rho)}{\rho}  \frac{\partial F}{\partial \rho}\sin \theta_n \right)\frac{\partial}{\partial y_n}
\end{align*}
where $r_n(r,\rho) = \sqrt{\rho^2 - \sum_{j=1}^{n-1}r_j^2}$. 
Hence,
\begin{align*}
    v_t(\theta,r) =& \nabla(\rho-f_t)(\theta,r,f_t(\theta,r))\\ =&\sum_{i=1}^{n-1}\left(\frac{1}{r_i}\frac{\partial f_t}{\partial \theta_i}(\theta,r)\sin \theta_i + \left(-\frac{\partial f_t}{\partial r_i}(\theta,r) + \frac{r_i}{f_t(\theta,r)}
    \right)\cos \theta_i \right)\frac{\partial}{\partial x_i} \\&+ \left(\frac{1}{r_n(t,\theta,r)}\frac{\partial f_t}{\partial \theta_n}(\theta,r)\sin \theta_n  +  \frac{r_n(t,\theta,r)}{f_t(\theta,r)}\cos \theta_n \right)\frac{\partial}{\partial x_n}  \\&+ \sum_{i=1}^{n-1}\left(-\frac{1}{r_i}\frac{\partial f_t}{\partial \theta_i}(\theta,r)\cos \theta_i + \left(-\frac{\partial f_t}{\partial r_i}(\theta,r) + \frac{r_i}{f_t(\theta,r)}
    \right)\sin \theta_i \right) \frac{\partial}{\partial y_i} \\&+ \left(-\frac{1}{r_n(t,\theta,r)}\frac{\partial f_t}{\partial \theta_n}(\theta,r)\cos \theta_n  + \frac{r_n(t,\theta,r)}{f_t(\theta,r)}  \sin \theta_n \right)\frac{\partial}{\partial y_n}
\end{align*}
where $r_n(t,\theta,r) = \sqrt{f_t(\theta,r)^2 - \sum_{j=1}^{n-1}r_j^2}$. From this we derive
\begin{align*}
    v_0(\theta,r) =&\sum_{i=1}^{n-1} \left(-\frac{\partial f_0}{\partial r_i}(r) + \frac{r_i}{f_0(r)}
    \right)\cos \theta_i \,\frac{\partial}{\partial x_i} +   \frac{r_n(0,r)}{f_0(r)}\cos \theta_n \,\frac{\partial}{\partial x_n}  \\&+ \sum_{i=1}^{n-1} \left(-\frac{\partial f_0}{\partial r_i}(r) + \frac{r_i}{f_0(r)}
    \right)\sin \theta_i \, \frac{\partial}{\partial y_i} + \frac{r_n(0,r)}{f_0(r)}\sin \theta_n \,\frac{\partial}{\partial y_n},
\end{align*}
and
\begin{align*}
    \dot{v}_0(\theta,r) =& \sum_{i=1}^{n-1}\left(\frac{1}{r_i}\frac{\partial \dot{f}_0}{\partial \theta_i}(\theta,r)\sin \theta_i + \left(-\frac{\partial \dot{f}_0}{\partial r_i}(\theta,r) - \frac{r_i \dot{f}_0(\theta,r)}{f^2_0(r)}
    \right)\cos \theta_i \right)\frac{\partial}{\partial x_i} \\&+ \left(\frac{1}{r_n(0,r)}\frac{\partial \dot{f}_0}{\partial \theta_n}(\theta,r)\sin \theta_n  +  \frac{(\sum_{j=1}^{n-1}r_j^2)}{r_n(0,r)}\frac{\dot{f}_0(\theta,r)}{f^2_0(r)}\cos \theta_n \right)\frac{\partial}{\partial x_n}  \\&+ \sum_{i=1}^{n-1}\left(-\frac{1}{r_i}\frac{\partial \dot{f}_0}{\partial \theta_i}(\theta,r)\cos \theta_i + \left(-\frac{\partial \dot{f}_0}{\partial r_i}(\theta,r) - \frac{r_i \dot{f}_0(\theta,r)}{f^2_0(r)}
    \right)\sin \theta_i \right)\frac{\partial}{\partial y_i} \\&+  \left(-\frac{1}{r_n(0,r)}\frac{\partial \dot{f}_0}{\partial \theta_n}(\theta,r)\cos \theta_n  +  \frac{(\sum_{j=1}^{n-1}r_j^2)}{r_n(0,r)}\frac{\dot{f}_0(\theta,r)}{f^2_0(r)}\sin \theta_n \right)\frac{\partial}{\partial y_n}.
\end{align*}
A straight forward computation yields 
\begin{lemma}\label{normv}
\begin{align*}
    \displaystyle\|v_0(\theta,r)\|^2 \,= \sum_{i=1}^{n-1}\left(\frac{r_i}{f_0(r)}-\frac{\partial f_0}{\partial r_i}(r) \right)^2 + \frac{r_n(0,r)^2}{f_0(r)^2}. 
\end{align*}
In particular, $\|v_0(\theta,r)\|$ only depends on $r$. 
\end{lemma}

Another simple computation yields.
\begin{align*}
    \displaystyle\langle v_0(\theta,r), \dot{v_0}(\theta,r) \rangle = \left(\frac{1}{f_0(r)^2}\sum_{i=1}^{n-1}r_i \frac{\partial f_0}{\partial r_i}(r)\right)\dot{f_0}(\theta,r) + \sum_{i=1}^{n-1}\left(\frac{\partial f_0}{\partial r_i}(r)-\frac{r_i}{f_0(r)}\right)\frac{\partial \dot{f_0}}{\partial r_i}(\theta,r).
\end{align*}
Together with Lemma \ref{2a}, this implies
\begin{lemma}\label{vdotv} The function 
    $$u \mapsto \displaystyle\langle v_0(w_K^{-1}(u)), \dot{v_0}(w_K^{-1}(u)) \rangle$$ is $\theta$-simple for any smooth $H \colon \RR^{2n}\to \RR$.
\end{lemma}

\begin{lemma}\label{Dv}
    The function $$u=(\theta,r) \mapsto \left\langle u, Dw_K^{-1}(u)\left[  v_0(w_K^{-1}(u))\right] \right\rangle$$ only depends on $r$.
\end{lemma}

\begin{proof}
Set \begin{align*}
C_k(r) =
    \begin{cases}
    \displaystyle-\frac{\partial f_0}{\partial r_i}(F(r)) + \frac{F_i(r)}{f_0(F(r))}\,\,\text{for}\,\,k \neq n\\\\
        \displaystyle \frac{r_n(0,F(r))}{f_0(F(r))} \,\,\text{for}\,\,k=n.
    \end{cases}
\end{align*}
Then $$v_0(w_K^{-1}(u)) =\sum_{k=1}^n C_k(r)\cos\theta_k \frac{\partial}{\partial x_k}+ \sum_{k=1}^n C_k(r)\sin\theta_k \frac{\partial}{\partial y_k}.$$ It follows from Lemma \ref{D} that
\begin{align*}Dw_K^{-1}(u)\left[  v_0(w_K^{-1}(u))\right] =& \sum_{j=1}^n \left( r_j\sum_{k=1}^{n-1}C_k(r)\frac{\partial a_j}{\partial r_k}(r)+ a_j(r)C_j(r) \right)\cos \theta_j \frac{\partial}{\partial x_j}\\ &+ \nonumber \sum_{j=1}^n \left( r_j\sum_{k=1}^{n-1}C_k(r)\frac{\partial a_j}{\partial r_k}(r)+ a_j(r)C_j(r) \right)\sin \theta_j \frac{\partial}{\partial y_j}.\end{align*}
Hence, we have
$$\langle u, Dw_K^{-1}(u)\left[  v_0(w_K^{-1}(u))\right] \rangle = \sum_{j=1}^n \left( r_j^2\sum_{k=1}^{n-1}C_k(r)\frac{\partial a_j}{\partial r_k}(r)+ r_j a_j(r)C_j(r) \right).$$

\end{proof}

\begin{lemma}\label{Dvdot}
    The function $$u=(\theta,r) \mapsto \left\langle u, Dw_K^{-1}(u)\left[  \dot{v}_0(w_K^{-1}(u))\right] \right\rangle$$ is $\theta$-simple for any smooth $H \colon \RR^{2n}\to \RR$.
\end{lemma}

\begin{proof}
Set \begin{align*}
A_k(\theta,r) =
    \begin{cases}
    \displaystyle\frac{1}{F_k(r)}\frac{\partial \dot{f}_0}{\partial \theta_k}(\theta,F(r))\,\,\text{for}\,\,k \neq n\\\\
        \displaystyle \frac{1}{r_n(0,F(r))}\frac{\partial \dot{f}_0}{\partial \theta_n}(\theta,F(r)) \,\,\text{for}\,\,k=n,
    \end{cases}
\end{align*}
and 
\begin{align*}
B_k(\theta,r) =
    \begin{cases}
    \displaystyle-\frac{\partial \dot{f}_0}{\partial r_k}(\theta,F(r)) - \frac{F_k(r) \dot{f}_0(\theta,F(r))}{f^2_0(F(r))}\,\,\text{for}\,\,k \neq n\\\\
        \displaystyle \frac{(\sum_{j=1}^{n-1}F_j(r)^2)}{r_n(0,F(r))}\frac{\dot{f}_0(\theta,F(r))}{f^2_0(F(r))} \,\,\text{for}\,\,k=n.
    \end{cases}
\end{align*}
Then 
\begin{align*}
    \dot{v}_0(w_K^{-1}(u)) =& \sum_{k=1}^{n} \left(A_k(\theta,r) \sin \theta_k + B_k(\theta,r)
    \cos \theta_k \right)\frac{\partial}{\partial x_k} \\&+ \sum_{k=1}^{n} \left(-A_k(\theta,r) \cos \theta_k + B_k(\theta,r)
    \sin \theta_k \right)\frac{\partial}{\partial y_k},
\end{align*}
and Lemma \ref{D} implies that $Dw_K^{-1}(u)\left[  \dot{v}_0(w_K^{-1}(u))\right]$ is equal to
\begin{align*}
 &{}\sum_{j=1}^n \left( a_j(r)A_r(\theta,r)\sin \theta_j  + \left(r_j\sum_{k=1}^{n-1}B_k(\theta,r) \frac{\partial a_j}{\partial r_k}(r) +a_j(r)B_j(\theta,r)\right)\cos \theta_j  \right)\frac{\partial}{\partial x_j}\\+&\sum_{j=1}^n \left( -a_j(r)A_r(\theta,r)\cos \theta_j  + \left(r_j\sum_{k=1}^{n-1}B_k(\theta,r) \frac{\partial a_j}{\partial r_k}(r) +a_j(r)B_j(\theta,r)\right)\sin \theta_j\right)\frac{\partial}{\partial y_j}.
\end{align*}
Hence,
$$
\left\langle u, Dw_K^{-1}(u)\left[  \dot{v}_0(w_K^{-1}(u))\right] \right\rangle = \left(\sum_{k=1}^{n-1}B_k(\theta,r) \frac{\partial a_j}{\partial r_k}(r)\right)r^2_j +r_ja_j(r)B_j(\theta,r)
$$
which is $\theta$-simple because the $B_j(\theta,r)$ are.

\end{proof}

We are now in the position complete the proof of Proposition \ref{form} by verifying that the three terms on the right side of equation \eqref{master} are all $\theta$-simple. 

\medskip

\noindent{\bf First term.} By equations \eqref{dinverse} and \eqref{w2v}, we have
\begin{align*}
 &f_0(w_K^{-1}(u))\left\langle u, \left.\frac{d}{dt}\right|_{t=0} w^{-1}_{K_t}(u)\right\rangle  
\\ =& -f_0(w_K^{-1}(u))\left\langle u,  Dw_K^{-1}(u)\left[  \dot{w}_K(w_K^{-1}(u))\right]\right\rangle \\ =&-\frac{f_0(w_K^{-1}(u))}{\|v_0(w_K^{-1}(u))\|}\left\langle u,  Dw_K^{-1}(u)\left[  \dot{v}_0(w_K^{-1}(u))\right]\right\rangle \\ &-\frac{f_0(w_K^{-1}(u))\langle\dot{v}_0(w_K^{-1}(u)),v_0(w_K^{-1}(u)) \rangle}{\|v_0(w_K^{-1}(u))\|^3}\left\langle u,  Dw_K^{-1}(u)\left[ v_0(w_K^{-1}(u))\right]\right\rangle.
\end{align*}
Lemmas \ref{f0}, \ref{normv} and Lemma \ref{Dvdot} imply that the first summand is $\theta$-simple, and Lemmas \ref{f0}, \ref{normv}, \ref{vdotv} and  \ref{Dv} imply that the second one is.

\medskip

\noindent{\bf Second term.}
By Lemma \ref{2a}, $\dot{f}_0(w_{K}^{-1}(u))$ is $\theta$-simple. Lemma \ref{2b} implies that $\langle u, w_K^{-1}(u)\rangle$ is a function of $r$. Hence their product, the second term in \eqref{master}, is $\theta$-simple.

\medskip

\noindent{\bf Third term.} By Lemma \ref{2b}, it suffices to show that the function
\begin{align*}
 df_0 (w_{K}^{-1}(u)) \left[  \left.\frac{d}{dt}\right|_{t=0} w^{-1}_{K_t}(u)\right]
\end{align*}
is $\theta$-simple. Invoking \eqref{w2v} and arguing as for the first term, it is enough to show the following
\begin{lemma}\label{nail}
The function 
\begin{align*}
     df_0 (w_{K}^{-1}(u)) \left[Dw_K^{-1}(u)\left[ v_0(w^{-1}_{K}(u))\right]\right]
\end{align*}
only depends on $r$, and the function
\begin{align*}
     df_0 (w_{K}^{-1}(u)) \left[Dw_K^{-1}(u)\left[ \dot{v}_0(w^{-1}_{K}(u))\right]\right]
\end{align*}
is $\theta$-simple.
\end{lemma}

\begin{proof}
    In the notation of the proof of Lemma \ref{Dv}, we have
    \begin{align*}Dw_K^{-1}(u)\left[  v_0(w_K^{-1}(u))\right] =& \sum_{j=1}^n \left( r_j\sum_{k=1}^{n-1}C_k(r)\frac{\partial a_j}{\partial r_k}(r)+ a_j(r)C_j(r) \right)\cos \theta_j \frac{\partial}{\partial x_j}\\ &+\nonumber \sum_{j=1}^n \left( r_j\sum_{k=1}^{n-1}C_k(r)\frac{\partial a_j}{\partial r_k}(r)+ a_j(r)C_j(r) \right)\sin \theta_j \frac{\partial}{\partial y_j}.\end{align*}
Since
\begin{align*}
    df_0(u) = \sum_{j=1}^{n-1} \frac{\partial f_0}{\partial r_j}(r) dr_j =\sum_{j=1}^{n-1} \frac{\partial f_0}{\partial r_j}(r) (\cos \theta_j\, dx_j + \sin \theta_j \,dy_j),
\end{align*}
we have
\begin{align*}
    df_0(w_K^{-1}(u)) =\sum_{j=1}^{n-1} \frac{\partial f_0}{\partial r_j}(F(r)) (\cos \theta_j\, dx_j + \sin \theta_j \,dy_j).
\end{align*}
Hence,  
\begin{align*}
     df_0 (w_{K}^{-1}(u)) \left[Dw_K^{-1}(u)\left[ v_0(w^{-1}_{K_t}(u))\right]\right] = \sum_{j=1}^n \frac{\partial f_0}{\partial r_j}(F(r))\left( r_j\sum_{k=1}^{n-1}C_k(r)\frac{\partial a_j}{\partial r_k}(r)+ a_j(r)C_j(r) \right)
\end{align*}
and the first assertion of Lemma has been verified. To verify the second, we use the notation of the proof of Lemma \ref{Dvdot} to write
\begin{align*}
 &Dw_K^{-1}(u)\left[  \dot{v}_0(w_K^{-1}(u))\right] \\=& \sum_{j=1}^n \left( a_j(r)A_r(\theta,r)\sin \theta_j  + \left(r_j\sum_{k=1}^{n-1}B_k(\theta,r) \frac{\partial a_j}{\partial r_k}(r) +a_j(r)B_j(\theta,r)\right)\cos \theta_j  \right)\frac{\partial}{\partial x_j}\\&+ \sum_{j=1}^n \left( -a_j(r)A_r(\theta,r)\cos \theta_j  + \left(r_j\sum_{k=1}^{n-1}B_k(\theta,r) \frac{\partial a_j}{\partial r_k}(r) +a_j(r)B_j(\theta,r)\right)\sin \theta_j\right)\frac{\partial}{\partial y_j}.
\end{align*}
It follows that 
\begin{align*}
     df_0 (w_{K}^{-1}(u)) \left[Dw_K^{-1}(u)\left[ \dot{v}_0(w^{-1}_{K_t}(u))\right]\right] = \sum_{j=1}^n  \frac{\partial f_0}{\partial r_j}(F(r))\left(r_j\sum_{k=1}^{n-1}B_k(\theta,r) \frac{\partial a_j}{\partial r_k}(r) +a_j(r)B_j(\theta,r)\right).
\end{align*}
Since the $B_j(\theta,r)$ are $\theta$-simple, the proof of Proposition \ref{form} is complete.
\end{proof}

Theorem \ref{variation} has now been established for a toric convex body which is strictly convex and has smooth boundary. An approximation argument will now allow us to relax these extra assumptions. 

Let $K$ be a toric convex body and consider a Hamiltonian flow, $\phi^t_H$. We need to prove that for the function $F(t) = \mw(\phi^t_H(K))$ we have 
$F'(0)=0$. The work of Weil in \cite{weil}, implies that there is a sequence of convex bodies $K_j$ converging to $K$ in the Hausdorff topology, such that the boundary of each $K_j$ is smooth, the principal curvatures of each $K_j$ are positive, and the principal curvatures of the $K_j$ converge to those of $K$ almost everywhere. Consider the functions $$F_j(t) = \mw(\phi^t_H (K_j))$$ for $t$ in some fixed interval $[-\epsilon, \epsilon]$. By the proof above, we have $F_j'(0)=0$ for all $j \in \NN$. Since the sequence $K_j$ converges to $K$ in the Hausdorff topology, the sequence $F_j$ converges to $F$ pointwise. The assertion about the principal  curvatures, together with the fact that $H$ has compact support, implies that $|F_j''(t)|$ is uniformly bounded for all $t\in [-\epsilon, \epsilon]$ and for all $j \in \NN$. Hence, the derivatives $F'_j$ converge uniformly. From this, it follows that the sequence $F_j$ converges uniformly to $F$ and that $$F'(0) =  \lim_{j\to \infty} F_j'(0) =0,$$as desired.

\section{Proof of Theorem \ref{linear}}\label{sec:linear}

\subsection{Reducing to the case of a convex body $K$ with smooth and strictly convex boundary} Arguing by approximation, as in the proof of Theorem \ref{variation} above, one can show that it suffices to prove that the identity matrix  $\II \in \mathrm{Sp}(2n)$ is an isolated local minimum of the function $\mathrm{Sp}(2n) \to \RR$ defined by 
\begin{align*}
 P \mapsto \mw(PK)
\end{align*}
when $K$ is a toric convex body which is strictly convex and has smooth boundary. We may also assume, in what follows, that the center of mass of $K$ is the origin.

\subsection{On the support function of $K$} The support function of $K$ is given by the formula $h_K(u) = \langle u,\nu^{-1}_{K}(u)\rangle$.  This function is smooth, as is its homogeneous extension $H_K \colon \RR^{2n} \setminus \{0\} \to \RR$ defined by 
\begin{align*}
    H_K(z)  &=\sup_{k \in K} \langle k,z\rangle \\
    {\,} &=\|z\| \, h_K \left( z/ \|z\|\right)\\
    {\,}     &=\left\langle z,\nu^{-1}_{K}\left(  z / \|z\|\right)\right\rangle.
\end{align*}
Denoting the components of $z \in \RR^{2n}$ by $$z=(z_1, \dots, z_{2n}) = (x_1, \dots, x_n, y_1, \dots, y_n),$$
we observe that the assumption that the center of mass of $K$ is the origin implies that
\begin{align}\label{eq:possumm}
    z_i (\nu^{-1}_{K}\left(  z / \|z\|\right))_i \geq 0, \text{ for all $i =1, \dots , 2n$},
\end{align}
where $(\nu^{-1}_{K}\left(  z / \|z\|\right))_i$ is the $i$th component of $\nu^{-1}_{K}\left(  z / \|z\|\right) \in \RR^{2n}$.


Since $H_K$ is a homogeneous function of degree 1, we have  $H_K(x) = \langle z,\nabla H_K(z)\rangle$. Hence, $$\nu^{-1}_{K}\left(  z / \|z\|\right) = \nabla H_K(z) $$
and \eqref{eq:possumm} implies that  
\begin{align}\label{eq:gradprts}
    x_i\frac{\partial H_K}{\partial x_i}(z)  \geq 0 \,\,\text{  and   }\,\,  y_i\frac{\partial H_K}{\partial y_i}(z)>0 \text{ for all $i =1, \dots , n$}.
\end{align}

Note also that for any $P \in \mathrm{Sp}(2n)$, the set $PK$ is convex and 
\begin{align}\label{eq;trans}
    h_{PK}(u) = H_K(P^T u).
\end{align}

\subsection{Restricting to symmetric positive definite symplectic matrices.} Every symplectic matrix $P \in \mathrm{Sp}(2n)$ has a unique polar decomposition, $P = QS$, 
where $Q$ belongs to $\mathrm{U}(n)$ and $S$ belongs to $Sym^+(\mathrm{Sp}(2n))$, the submanifold of symmetric positive definite symplectic matrices. It then follows from Remark \ref{toric} that in order to prove Theorem \ref{linear} it suffices to show that the identity matrix  $\II $  is an isolated  local minimum of the function $Sym^+(\mathrm{Sp}(2n)) \to \RR$ defined by 
\begin{align*}
S \mapsto \mw(SK).
\end{align*}
In particular, by formula \eqref{eq;trans}, it suffices to prove the following.

\begin{proposition}\label{prop:lin}
    For any smooth path $A(s)$ in $Sym^+(\mathrm{Sp}(2n))$ with $A(0)=\II$ and $A'(0) \neq 0$, the function
    \begin{align}\label{eq:calc}
        f(s) = \int_{S^{2n-1}} H_K(A(s)u) \, d\sigma
    \end{align}
    satisfies $f'(0) =0$ and $f''(0) >0.$
\end{proposition}

\subsection{Proof of Proposition \ref{prop:lin}}

The fact that $f'(0) =0$  for any choice of the path $A(s)$ is a direct consequence of Theorem \ref{variation}. In particular, we have
\begin{align}\label{eq:zero}
    \int_{S^{2n-1}} \langle \nabla H_K(u), Y u \rangle \, d\sigma =0
\end{align}
for any matrix $Y$ of the form  
\begin{align*}
    Y = \begin{pmatrix}
    C & D\\
    D & -C
    \end{pmatrix}, 
\end{align*}
where the submatrices $C$ and $D$ are both symmetric.

To prove that $f''(0)>0$ we first note that  
\begin{align}\label{eq:f''}
    f''(s) = \int_{S^{2n-1}} \text{Hess}(H_K)(A(s)u)(A'(s)u, A'(s)u) \, d\sigma + \int_{S^{2n-1}} \langle \nabla H_K(A(s)u), A''(s)u \rangle \,d\sigma. 
\end{align}
Here, $\text{Hess}(H_K)$ denotes the Hessian of $H_K$ which is well-defined on $\RR^{2n} \setminus\{0\}$. Since $K$ is assumed to be strictly convex, the Hessian is positive definite. Thus,  since $A'(0) \neq 0$, the first summand on the right of \eqref{eq:f''} is positive when $s=0$, and it  suffices to prove that 
\begin{align*}
\int_{S^{2n-1}} \langle \nabla H_K(u), A''(0)u \rangle \,d\sigma
\end{align*}
is nonnegative.

We can write $A'(s) = A(s)X(s)$ for a family of symmetric matrices $X(s) \in \mathfrak{sp}(2n)$ of the form 
\begin{align*}
    X(s) = \begin{pmatrix}
    C(s) & D(s)\\
    D(s) & -C(s)
    \end{pmatrix},
\end{align*}
where the submatrices $C(s)$ and $D(s)$ are all symmetric. We then have $A''(s) = A(s) (X(s)^2 + X'(s))$ and $A''(0) = X(0)^2 + X'(0)$. Applying $\eqref{eq:zero}$ to $Y=X'(0)$ we have
\begin{align*}
    \int_{S^{2n-1}} \langle \nabla H_K(u), X'(0)u \rangle \, d\sigma = 0.
\end{align*}
It remains to show that 
\begin{align*}
    \int_{S^{2n-1}} \langle \nabla H_K(u), X(0)^2u \rangle d\sigma \geq 0.
\end{align*}
The expression above implies that 
\begin{align*}
    X(0)^2 =&\begin{pmatrix}
    C(0)^2+D(0)^2 & C(0)D(0) - D(0)C(0)\\
    D(0)C(0) - C(0)D(0) & C(0)^2 +D(0)^2
    \end{pmatrix}=:    \begin{pmatrix}
    L & M\\
    -M & L
    \end{pmatrix},
\end{align*}
and it is straight forward to check that the submatrix $L =(\ell_{ij})$ is symmetric, the submatrix $M=(m_{ij})$ is skew-symmetric and the diagonal entries of $L$ are all nonnegative. Setting $$\nabla H_K(u) = (a(u), b(u)) \in \RR^n \times \RR^n,$$ we then have
\begin{align*}
     \int_{S^{2n-1}} \langle \nabla H_K(u), X(0)^2u \rangle \, d\sigma =& \sum_i  \int_{S^{2n-1}} \ell_{ii}(a_i(u)x_i + b_i(u)y_i) \, d\sigma\\&+ \sum_{i<j}  \int_{S^{2n-1}} \ell_{ij} (a_i(u)x_j +b_i(u)y_j +a_j(u)x_i +b_j(u)y_i) \, d\sigma\\& + \sum_{i<j}  \int_{S^{2n-1}} m_{ij}(a_i(u)y_j-b_i(u)x_j - a_j(u)y_i + b_j(u)x_i) \, d\sigma
\end{align*}
Since the $\ell_{ii}$ are nonnegative, it follows from \eqref{eq:gradprts} that the first term above is nonnegative. The fact that $K$ is toric imples that each of the other terms vanish and hence we are done. To see this consider, for example, the terms $\int_{S^{2n-1}}a_i(u)x_j\, d\sigma$ with $i<j.$ Since $i \neq j$ and $K$ is toric, the function $a_i(u) = \frac{\partial H_K}{\partial x_i}(u)$ is invariant under the change of variables $x_j \mapsto -x_j$. It follows that $$\int_{S^{2n-1}}a_i(u)x_j\, d\sigma =-\int_{S^{2n-1}}a_i(u)x_j\, d\sigma =0.$$
This completes the proof.

\section{Proof of Proposition \ref{ellipsoid}}\label{sec:ellipsoid}

Given Theorem \ref{linear}, it suffices to prove that if $\mathbf{a} \neq \mathbf{b}$, then there exist a smooth path $S(s)$ in $Sym^+(\mathrm{Sp}(2n))$ with $S(0) = \II$ such that $$\left.\frac{d}{ds}\right|_{s=0} M(S(s) E(\mathbf{a},\mathbf{b})) \neq 0.$$

We begin by analyzing the formula for $M(S(s) E(\mathbf{a},\mathbf{b}))$.
For any matrix $T \in \mathrm{GL}(d),$ the support function of the ellipsoid $TB^d$ is given by $$h_{TB^d}(u) = \|T^t u\|.$$ Since $E(\mathbf{a},\mathbf{b}) = \Delta(\mathbf{a}, \mathbf{b})B^{2n}$, for $\Delta(\mathbf{a}, \mathbf{b}) = \mathrm{diag}(a_1, \dots, a_n, b_1, \dots, b_n),$ the formula for $M(S(s) E(\mathbf{a},\mathbf{b}))$ can be simplified to 
\begin{align*}
   M(S(s) E(\mathbf{a},\mathbf{b}))  = 2\int_{S^{2n-1}} \|\Delta(\mathbf{a}, \mathbf{b}) S(s) u\| d\sigma.
\end{align*}

 Relabelling if necessary, we may assume that $a_1 < b_1$. Set $$S(s) = \II + (e^s-1)\II_{1,1}+(e^{-s}-1)\II_{n+1,n+1}$$ where $\II_{i,j}$ is the matrix whose $(i,j)-$th entry is one and whose other entries are zero. We then have 
 \begin{align*}
    \left.\frac{d}{ds}\right|_{s=0} M(S(s) E(\mathbf{a},\mathbf{b})) &= 2\int_{S^{2n-1}} \frac{\langle \Delta(\mathbf{a}, \mathbf{b})u, \Delta(\mathbf{a}, \mathbf{b}) S'(0) u \rangle }{\|\Delta(\mathbf{a}, \mathbf{b})u\|}\,d\sigma \\&= 2\int_{S^{2n-1}} \frac{a_1^2x_1^2-b_1^2y_1
^2 }{\|\Delta(\mathbf{a}, \mathbf{b})u\|}\,d\sigma,
\end{align*}
which we claim is negative. In Hopf coordinates,
$$\frac{a_1^2x_1^2-b_1^2y_1
^2 }{\|\Delta(\mathbf{a}, \mathbf{b})u\|} = \frac{r_1^2(a_1^2\cos^2\theta_1- b_1^2\sin^2\theta_1)}{\left(r_1^2(a_1^2\cos^2\theta_1+ b_1^2\sin^2\theta_1)+\sum_{j = 2}^n r_j^2(a_j^2\cos^2\theta_j+ b_j^2\sin^2\theta_j) \right)^{1/2}}.$$
Define $c>0$ by $(1+c)a_1^2 = b_1^2$, and set $A_1= r_1^2a_1^2$ and $B_1 = \sum_{j =2}^n r_j^2(a_j^2\cos^2\theta_j+ b_j^2\sin^2\theta_j)$. Considering the integral over  $\theta_1$ first, it suffices to show that the function
\begin{align*}
    I(c) 
    &= \bigintssss_0^{2 \pi} \frac{A_1(\cos^2\theta_1- (1+c)\sin^2\theta_1)}{\left(A_1(\cos^2\theta_1+ (1+c)\sin^2\theta_1)+B_1 \right)^{1/2}} \,\, d \theta_1
\end{align*}
 is negative for all $c>0$. Clearly, $I(0)=0$ and a simple computation yields  
\begin{align*}
    I'(c) &= -\bigintssss_0^{2 \pi} \frac{A_1 \sin^2\theta_1\left[\frac{3A_1}{2}\cos^2\theta_1+ (1+c)\frac{A_1}{2}\sin^2\theta_1+B_1\right]}{\left(A_1(\cos^2\theta_1+ (1+c)\sin^2\theta_1)+B_1 \right)^{3/2}} \,\, d \theta_1.
\end{align*}
This is clearly negative for all $c>0$, and so the proof is complete.

\section{Proofs of Proposition \ref{lag} and Proposition \ref{vr}}\label{sec:lag}

Here we prove the two results concerning the Lagrangian bidisk $$\mathbf{P}_L = \left\{ (x_1, x_{2}, y_1, y_{2}) \in \mathbb{R}^{4} \bigmid  x_1^2 +x_2^2 \leq 1,\, y_1^2 +y_2^2 \leq 1\right\}.$$

\subsection{The proof of Proposition \ref{lag}} To prove the first assertion it suffices to show that $\II$ is an isolated local minimum of the function  $\mathcal{M}_{\mathbf{P}_L} \colon Sym^+(\mathrm{Sp}(4)) \to \mathbb{R}$ defined by
\begin{align*}
    \mathcal{M}_{\mathbf{P}_L}(S) = \int_{S^{3}} h_{\mathbf{P}_L}(S u) d\sigma.
\end{align*}

For  $u \in S^{3} \subset \RR^{4}$ let $u_x$ be the projection to the $x_1x_2$-plane and $u_y$ be the projection to the $y_1y_2$-plane. We then have $h_{\mathbf{P}_L}(u) =  \|u_x\| + \|u_y \|.$ For any smooth curve $S(s)$ in $Sym^+(\mathrm{Sp}(4))$ with $S(0) = \II$  we have 
\begin{equation}\label{critpl}
    \left.\frac{d}{ds}\right|_{s=0} \mathcal{M}_{\mathbf{P}_L} (S(s)) = \int_{S^{3}} \left\langle \frac{u_x}{\|u_x\|}+\frac{u_y}{\|u_y\|}, S'(0)u \right\rangle d\sigma,
\end{equation}
since  $\nabla h_{\mathbf{P}_L}(u)=\frac{u_x}{\|u_x\|}+\frac{u_y}{\|u_y\|}$ holds almost everywhere.
Setting $S'(0) = \begin{pmatrix} C & D\\ D & -C \end{pmatrix}$ we have
\begin{align*}
    \langle u_x, S'(0)u \rangle &= c_{11}x_1^2 +2c_{12}x_1x_2 + c_{22}x_2^2 +d_{11}x_1y_1+d_{12}(x_1y_2+ x_2y_1) + d_{22}x_2y_2
\end{align*}
and
\begin{align*}
    \langle u_y, S'(0)u \rangle &= d_{11}x_1y_1+d_{12}(x_1y_2+ x_2y_1) + d_{22}x_2y_2 - c_{11}y_1^2 -2c_{12}y_1y_2 - c_{22}y_2^2.
\end{align*}
We now use Lagrangian Hopf coordinates on $\RR^{4}$, $(r, \theta) = (r_1, r_{2}, \theta_1, \theta_{2}),$ defined by $$(x_1,x_2) = (r_1 \cos{\theta_1}, r_1 \sin{\theta_1}), \text{ and  }(y_{1},y_{2}) = (r_{2} \cos{\theta_{2}}, r_{2} \sin{\theta_{2}}).$$ In these coordinates it is easy to show, as before, that the terms 
in \eqref{critpl} with integrands corresponding to cross-terms all vanish.  We are then left with
\begin{align*}
    \left.\frac{d}{ds}\right|_{s=0} \mathcal{M}_{\mathbf{P}_L} (S(s)) = \int_{S^{3}} \left( \sum_i^{2}\frac{ c_{ii}x_i^2}{\|u_x\|}-\sum_i^{2}\frac{ c_{ii}y_i^2}{\|u_y\|}\right) d\sigma
\end{align*}
which vanishes, this time by symmetry. Hence, the identity matrix is a critical point.

Next we verify that the second variation $\left.\frac{d^2}{ds^2}\right|_{s=0} \mathcal{M}_{\mathbf{P}_L} (S(s)),$ is positive.
Here we have  $S'(s)=S(s)X(s)$ for a family of symmetric matrices $X(s) \in \mathfrak{sp}(4)$ of the form 
\begin{align*}
    X(s) = \begin{pmatrix}
    C(s) & D(s)\\
    D(s) & -C(s)
    \end{pmatrix}
\end{align*}
where $C(s)$ and $D(s)$ are symmetric and $X(0) \neq 0$. A straightforward computation, together with the Cauchy-Schwarz inequality, then yields 
\begin{align}\label{eq:upper1}
  \left.\frac{d^2}{ds^2}\right|_{s=0} \mathcal{M}_{\mathbf{P}_L} (S(s))  \geq \int_{S^{3}}
\left\langle \frac{u_x}{\|u_x\|}+\frac{u_y}{\|u_y\|}, X(0)^2u \right\rangle d\sigma.
\end{align}
As before, 
\begin{align*}
    X(0)^2 &=  \begin{pmatrix}
    L & M\\
    -M & L
    \end{pmatrix},
    \end{align*}
where $L = C(0)^2 + D(0)^2$ is symmetric and has positive diagonal entries and $M=C(0)D(0)-D(0)C(0)$ is skew-symmetric. With this, the right-hand side of \eqref{eq:upper1} simplifies to
\begin{align*}
\int_{S^{3}}
\left\langle \frac{u_x}{\|u_x\|}+\frac{u_y}{\|u_y\|}, X(0)^2u \right\rangle d\sigma  = \int_{S^{3}} \left( \sum_i^{2}\frac{ l_{ii}x_i^2}{\|u_x\|}+\sum_i^{2n}\frac{ l_{ii}y_i^2}{\|u_y\|}\right) d\sigma,
\end{align*}
which is clearly positive, as desired.

To complete the proof of Theorem \ref{lag} it remains to verify that $\mathbf{P}_L$ is not equal to  $QX$ for any toric subset $X$ and any $Q \in \mathrm{U}(2)$.
For the orthogonal matrix\begin{equation*}
O = 
\begin{pmatrix}
1 & 0 & 0 &0 \\
0 & 0 & 1 &0 \\
0 & 1 & 0 &0 \\
0 & 0 & 0 &1 
\end{pmatrix}
\end{equation*} 
the set
$$O\mathbf{P}_L = \left\{ (x_1, x_{2}, y_1, y_{2}) \in \mathbb{R}^{4} \bigmid  x_1^2 + y_1^2 \leq 1,\, x_2^2 + y_2^2  \leq 1\right\}$$ is toric.
If we assume $\mathbf{P}_L=QX$ for some toric subset $X$ and some unitary matrix $Q \in \mathrm{U}(2)$, then the matrix $OQ$ would map the toric domain $X$ to another toric domain. This would imply that $OQ$ lies in $\mathrm{GL}(2,\CC) \cap \mathrm{O}(4) = \mathrm{U}(2)$ which would contradict the fact that $\det OQ =\det O =-1$.

\subsection{The proof of Proposition \ref{vr}}

We first recall the theorem, of Ramos, from \cite{vr} that underlies the proof. Let $\Omega_0 \subset \RR^2_{\geq 0}$ be the domain bounded by the coordinate axes and the curve 
$$
\left( 2\sin \left( \frac{\alpha}{2}\right)-\alpha \cos \left( \frac{\alpha}{2}\right),\, 2\sin \left( \frac{\alpha}{2}\right)+(2\pi-\alpha) \cos \left( \frac{\alpha}{2}\right) \right), \quad \alpha \in [0,2\pi].
$$
Let $X_0 = \mu^{-1}(\Omega_0)$ where $\mu \colon \RR^4 \to \RR^2_{\geq 0}$ is the standard moment map. The domain $X_0$ is toric and convex. It's boundary is not smooth.

\begin{theorem}[Ramos, \cite{vr} Theorem 3] There is a symplectic embedding $\phi_R$ from the interior of the Lagrangian bidisk, $\mathrm{int}(\mathbf{P}_L)$, into  $\RR^4$ such that $$\phi_R(\mathrm{int}(\mathbf{P}_L)) = \mathrm{int(X_0)}.$$
\end{theorem}

For any real number $\delta$, let $\eta_{\delta} \colon \RR^{2n} \to \RR^{2n}$ be multiplication by $e^{\delta}$. Since $\mathbf{P}_L$ is star-shaped, for every $\delta>0$ the map 
$$\eta_{\delta}\circ \phi_R \circ \eta_{-\delta} \colon \mathbf{P}_L \to \RR^{4}$$ is a symplectic embedding. By the {\em Extension after Restriction Principle}, \cite{schl}, there is a Hamiltonian diffeomorphism $\phi_{\delta} \in \Symp$ such that $$ \phi_{\delta}|_{\mathbf{P}_L} =(\eta_{\delta}\circ \phi_R \circ \eta_{-\delta}) |_{\mathbf{P}_L}.$$
Moreover, for any $\epsilon>0$ we can choose $\delta >0$ so that $d_H(\phi_{\delta}(\mathbf{P}_L), X_0) < \epsilon.$ Since the mean width is continuous in the Hausdorff topology, to prove Proposition \ref{vr} it suffices to prove
\begin{align*}
    \mw(X_0) < \mw(\mathbf{P}_L).
\end{align*}
A simple computation yields $\mw(\mathbf{P}_L) = \frac{8}{3}$. So, it suffices to construct a domain $X_1$ such that $X_0 \subset X_1$ and $\mw(X_1) < \frac{8}{3}.$ A simple toric approximation of $X_0$ does the job. Let $\Omega_1 \subset \RR^2_{\geq 0}$ be the domain bounded by the coordinate axes and the broken line whose first segment connects $(0,2\pi)$ to $(2,2)$ and whose second line segment connects $(2,2)$ to $(2\pi,0)$, see Figure \ref{fig:O1}.
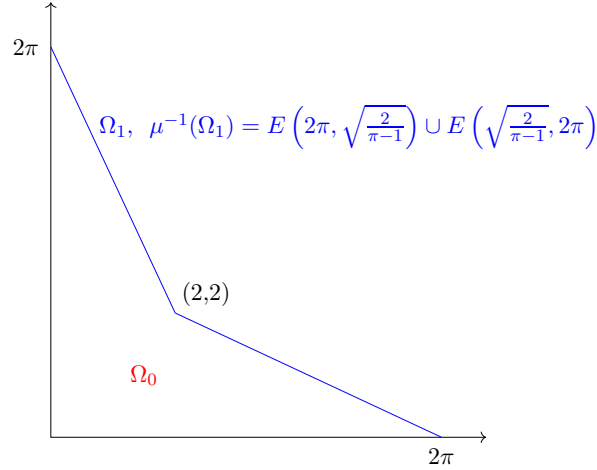
\begin{figure}
    \centering
    \caption{Approximating $X_0 =\mu^{-1}(\Omega_0)$ by $X_1 =\mu^{-1}(\Omega_1)$}
    \label{fig:O1}
        \vspace{.5cm}
    \resizebox{0.5\textwidth}{!}{
    \begin{tikzpicture}
     \tikzstyle{every node}=[font=\small]
    \draw[<->](0,7) -- (0,0) -- (7,0);
    \draw[blue] (0, 2*pi) -- (2, 2);
    \draw[blue] (2, 2) -- (2*pi, 0);
    \node at (2.5, 2.3) {(2,2)};
    \node at (-0.4, 2*pi) {2$\pi$};
    \node at (2*pi, -0.3) {2$\pi$};
    \draw[red, smooth,samples=500,domain=0:6.283185307179586] plot[parametric] function{2*sin((t/2))-t*cos((t/2)),2*sin((t/2))+(2*3.141592653589793-t)*cos((t/2))};
    \node[blue] at (4.8,5) {$\Omega_1,\,\,\,\mu^{-1}(\Omega_1) = E\left(2 \pi, \sqrt{\frac{2}{\pi -1}}\right) \cup E\left(\sqrt{\frac{2}{\pi -1}}, 2\pi\right)$};
    \node[red] at (1.5,1) {$\Omega_0$}; 
\end{tikzpicture}
}
\end{figure}
Since $\Omega_1$ contains $\Omega_0$, the toric domain $X_1 = \mu^{-1}(\Omega_1)$ contains $X_0$. Note that $X_1$ is a union of symplectic ellipsoids of the form $\mathbf{E}(a,b) \cup \mathbf{E}(b,a),$ for $a=\sqrt{2}$ and $b=\sqrt{\frac{2}{\pi-1}}.$ A straight forward computation yields $$\mw(\mathbf{E}(a,b) \cup \mathbf{E}(b,a)) = \frac{8}{3}\frac{1}{a^2-b^2}\left(a^3 -\left(\frac{a^2+b^2}{2}\right)^{\frac{3}{2}}\right).$$
From this we derive the desired inequality,  $$\mw(X_0)< \mw(X_1) \approx 2.63062 < \frac{8}{3}=\mw(\mathbf{P}_L).$$

\end{document}